\documentclass[12pt]{amsart}

\newtheorem{theorem}{Theorem}[section]
\newtheorem{proposition}[theorem]{Proposition}
\newtheorem{lemma}[theorem]{Lemma}
\newtheorem{corollary}[theorem]{Corollary}

\theoremstyle{definition}
\newtheorem{definition}[theorem]{Definition}

\newtheorem{example}[theorem]{Example}

\newtheorem{question}[theorem]{Question}
\newtheorem{conjecture}[theorem]{Conjecture}
\newtheorem{remark}[theorem]{Remark}

\newcommand{\ZZ}{ \ensuremath{\mathbb{Z}}}

\def\cocoa{{\hbox{\rm C\kern-.13em o\kern-.07em C\kern-.13em o\kern-.15em A}}}

\newcommand{\aaa}{\ensuremath{\mathbf{a}}}
\newcommand{\bb}{\ensuremath{\mathbf{b}}}
\newcommand{\cc}{\ensuremath{\mathbf{c}}}
\newcommand{\dd}{\ensuremath{\mathbf{d}}}
\newcommand{\cd}{\ensuremath{\mathbf{cd}}}
\newcommand{\sd}{\ensuremath{\mathcal{O}}}
\newcommand{\U}{\ensuremath{\mathcal{U}}}
\newcommand{\Z}{\ensuremath{\mathcal{Z}}}
\newcommand{\lk}{\mathrm{lk}}

\begin{document}

\title{The flag $f$-vectors of Gorenstein* order complexes of dimension $3$}

\author{Satoshi Murai}
\address{
Satoshi Murai,
Department of Mathematical Science,
Faculty of Science,
Yamaguchi University,
1677-1 Yoshida, Yamaguchi 753-8512, Japan
}
\email{murai@yamaguchi-u.ac.jp}

\author{Eran Nevo}
\address{
Department of Mathematics,
Ben Gurion University of the Negev,
Be'er Sheva 84105, Israel
}
\email{nevoe@math.bgu.ac.il}

\thanks{
Research of the first author was partially supported by KAKENHI 22740018.
Research of the second author was partially supported by Marie Curie grant IRG-270923.
}



\maketitle

\begin{abstract}
We characterize the $\cd$-indices of Gorenstein* posets of rank $5$, equivalently the flag $f$-vectors of Gorenstein* order complexes of dimension $3$.
As a corollary, we characterize the $f$-vectors of Gorenstein* order complexes in dimensions $3$ and $4$.
This characterization rise a speculated intimate connection between the $f$-vectors of flag homology spheres and the $f$-vectors of Gorenstein* order complexes.
\end{abstract}

\section{Introduction}
The flag $f$-vector is an important invariant of graded posets.
For a Gorenstein* poset $P$,
its flag $f$-vector is efficiently encoded by its $\cd$-index \cite{Bayer-Klapper} whose integer coefficients turn out to be non-negative \cite{ Karu-cd,Stanley-cd}.
Further restrictions on the $\cd$-index of Gorenstein* posets were obtained recently in \cite{Murai-Nevo-S*}, however a full characterization is not even conjectured yet.
In this paper we make a first step in this direction
by
characterizing the $\cd$-indices of Gorenstein* posets of rank $5$, which is the first nontrivial case.

We first recall the definition of the $\cd$-index.
Let $P$ be a graded poset of rank $n+1$ with the minimal element $\hat 0$ and the maximal element $\hat 1$.
The {\em order complex} $\sd(P)$ of $P$ (or of $P -\{\hat 0,\hat 1\}$),
is the (abstract) simplicial complex whose faces are the chains of $P-\{\hat 0,\hat 1\}$.
Thus
$$\sd(P)=\{\{\sigma_1,\sigma_2,\dots,\sigma_k\} \subseteq P-\{\hat 0,\hat 1\}: \sigma_1< \sigma_2 < \cdots < \sigma_k\}.$$
Let $r: P \to \ZZ_{\geq 0}$ denote the rank function of $P$.
For $S \subseteq [n]=\{1,2,\dots,n\}$,
an element $\{\sigma_1,\dots,\sigma_k\} \in \sd(P)$ with $\{r(\sigma_1),\dots,r(\sigma_k)\}=S$
is called an {\em $S$-flag} of $P$.
Let $f_S(P)$ be the number of $S$-flags of $P$.
Define $h_S(P)$ by
$$h_S(P)=\sum_{T \subseteq S} (-1)^{|S|-|T|} f_T (P),$$
where $|X|$ denotes the cardinality of a finite set $X$.
The vectors $(f_S(P): S \subseteq [n])$ and $(h_S(P): S \subseteq [n])$ are
called the \textit{flag $f$-vector}  and \textit{flag $h$-vector} of $P$ (or $\sd (P)$) respectively.
It is convenient to represent flag $h$-vectors as coefficients of non-commutative polynomials.
For $S \subseteq [n]$,
we define a non-commutative monomial $u_S=u_1u_2\cdots u_n$ in variables $\aaa$ and $\bb$
by $u_i=\aaa$ if $i \not \in S$ and $u_i=\bb$ if $i \in S$,
and define
$$\Psi_P(\aaa,\bb)= \sum_{S \subseteq [n]} h_S(P) u_S.$$
We say that $P$ is {\em Gorenstein*} if the simplicial complex $\sd(P)$ is Gorenstein* \cite[p.\ 67]{StanleyGreenBook}.
A typical example of a Gorenstein* poset comes from CW-spheres, namely, a regular CW-complex which is homeomorphic to a sphere.
Indeed, if $P$ is the face poset of a CW-sphere
then $P \cup \{\hat 0, \hat 1\}$ is a Gorenstein* poset.
It is known that if $P$ is Gorenstein* then $\Psi_P(\aaa,\bb)$ can be written as a polynomial $\Phi_P(\cc,\dd)$ in $\cc=\aaa+\bb$
and $\dd=\aaa \bb + \bb \aaa$ \cite{Bayer-Klapper},
and this non-commutative polynomial $\Phi_P(\cc,\dd)$ is called the \textit{$\cc\dd$-index} of $P$.

It is known that the coefficients of $\Phi_P(\cc,\dd)$ are non-negative integers
\cite{Karu-cd,Stanley-cd}
and the coefficient of $\cc^n$ in $\Phi_P(\cc,\dd)$ is $1$.
The main result of this paper is the next result,
which characterizes all possible $\cd$-indices of Gorenstein* posets of rank $5$.

\begin{theorem}\label{thm:cd-rank5}
The $\cd$-polynomial $\cc^4+ \alpha_1 \dd \cc^2 + \alpha_2 \cc\dd\cc + \alpha_3 \cc^2 \dd + \alpha_{13} \dd^2 \in \ZZ_{\geq 0}\langle \cc,\dd\rangle$
is the $\cd$-index of a Gorenstein* poset of rank $5$ if and only if one of the following conditions holds:
\begin{itemize}
\item[(i)] $\alpha_2=0$ and $\alpha_{13}=\alpha_1 \alpha_3$.
\item[(ii)] $\alpha_2=1$ and there are non-negative integers $b_1,b_2,b_3,c_1,c_2,c_3$ such that
$\alpha_1=b_1+b_2+b_3$,
$\alpha_3=c_1+c_2+c_3$ and $\alpha_{13}=\alpha_1\alpha_3 - (b_1 c_1 + b_2 c_2 + b_3c_3)$.
\item[(iii)] $\alpha_2 \geq 2$ and $\alpha_{13} \leq \alpha_1 \alpha_3$.
\end{itemize}
\end{theorem}

Note that,
since knowing the $\cd$-index of $P$
is equivalent to knowing the flag $f$-vector of $\sd(P)$,
Theorem \ref{thm:cd-rank5} characterizes the flag $f$-vectors of Gorenstein* order complexes of dimension $3$.

Recall that,
for a Gorenstein* poset $P$ of rank $n+1$,
the vector $(h_0,h_1,\dots,h_n)$, where $h_i=\sum_{S \subseteq [n], |S|=i} h_S(P)$,
is the usual $h$-vector of $\sd(P)$.
Thus,
as an immediate corollary of Theorem \ref{thm:cd-rank5},
we obtain a characterization of the $f$-vectors of Gorenstein* order complexes of dimension $3$.
We extend this later result to dimension $4$ as well.

The numerical conditions  are conveniently given in terms of $\dd$-vectors of Gorenstein* posets.
For a Gorenstein* poset $P$ of rank $n+1$, we define its {\em $\dd$-vector} $\dd(P)=(\delta_0,\delta_1,\dots,\delta_{\lfloor \frac n 2 \rfloor})$
by
$$\Phi_P(1,\dd)= \delta_0 + \delta_1 \dd + \cdots + \delta_{\lfloor \frac n 2 \rfloor} \dd ^{\lfloor \frac n 2 \rfloor},$$
where $\lfloor x \rfloor$ is the integer part of $x$.
Thus $\delta_0,\delta_1,\dots,\delta_{\lfloor \frac n 2 \rfloor}$
are the coefficients of the polynomial
which is obtained from  the $\cd$-index of $P$ by substituting $\cc=1$.
Note that, by the definition,
$\dd(P)$ is a non-negative vector with $\delta_0=1$.
Since the $\dd$-vector of $P$ and the $h$-vector $h=(h_0,h_1,\dots,h_n)$ of $\sd(P)$
 are related by 
$\sum_{i=0}^n h_i x^{i} = \sum_{i=0}^{\lfloor \frac n 2 \rfloor} 2^{i} \delta_i  x^i(1+x)^{n-2i}$ (cf.\ \cite[Section 4]{Murai-Nevo-S*}),
knowing the $\dd$-vector of $P$ is equivalent to knowing the $h$-vector of $\sd (P)$, equivalently the $f$-vector of $\sd (P)$.

\begin{theorem}\label{thm:rank5,6d-vectors}
Let $(1,x,y) \in \ZZ_{\geq 0}^3$.
\begin{itemize}
\item[(a)]
The vector $(1,x,y)$
is the $\mathbf{d}$-vector of a Gorenstein* poset of rank $5$ if and only if
it satisfies $y \leq {(x-1)^2 \over 4}$ or there are non-negative integers $a$ and $b$ such that $x=a+b$ and $y=ab$.
\item[(b)]
The vector $(1,x,y)$
is the $\mathbf{d}$-vector of a Gorenstein* poset of rank $6$ if and only if
it satisfies $y \leq {x^2 \over 4}$.
\end{itemize}
\end{theorem}

We notice that the vectors in part (b) are exactly the $\gamma$-vectors of homology $4$-spheres (details will be given in Section \ref{sec:Q}), and that the vectors in part (a) are conjectured in \cite{Gal} to be exactly the $\gamma$-vectors of homology $3$-spheres.
This gives rise to Question \ref{Q1}, about an interesting relation between $f$-vectors of flag homology spheres and those of Gorenstein* order complexes.



Outline of the paper: in Section \ref{sec:unzip} we describe a poset construction, which we call \emph{unzipping}.
Unzipping is in a certain sense an inverse of zipping \cite{Reading} and will be used to prove sufficiency in Theorems \ref{thm:cd-rank5} and \ref{thm:rank5,6d-vectors}; Theorem \ref{thm:cd-rank5} is proved in Section \ref{sec:cd-5}; $\dd$-vectors are discussed in Section \ref{sec:d-vectors} where Theorem \ref{thm:rank5,6d-vectors} is proved; in Section \ref{sec:Q} we discuss questions relating $f$-vectors of flag homology spheres with those of Gorenstein* order complexes.

\section{Unzipping}\label{sec:unzip}

We introduce a construction on graded posets which we call \emph{unzipping}. Post composition with the corresponding \emph{zipping}, as defined by Reading \cite{Reading}, gives back the original poset.

\begin{definition}\label{def:unzip}
Let $P$ be a graded poset with $\hat 0$ and $\hat 1$,
and let $r:P\rightarrow \ZZ_{\geq 0}$ be the rank function of $P$.
\begin{itemize}
\item[(1)] (Reading \cite{Reading})
Let $x,y,z\in P-\{\hat 0,\hat 1\}$ be such that (i) $x$ covers exactly $y$ and $z$,
(ii) $x$ is the unique minimal upper bound of $y$ and $z$,
and (iii) $y$ and $z$ cover exactly the same elements.
Let $\Z(P;x,y,z)$  be the poset obtained from $P$ by deleting $x,y$ and adding the relations $w>z$ for all relations $w>y$.

\item[(2)]
Let $x,y\in P -\{\hat 0,\hat 1\}$ be such that $x$ covers $y$.
We define the graded poset $\U (P;x,y)$ as follows:
delete the cover relation $y<x$, add elements $x',y'$ with ranks $r(x')=r(x), r(y')=r(y)$ and add cover relations (i) $x'<w$ for all covers $x<w$, (ii) $w<y'$ for all covers $w<y$ and (iii) $y'<x', y<x'$ and $ y'<x$.
\end{itemize}
The operations $P \to \Z(P;x,y,z)$ and $P \to \U(P;x,y)$ are called {\em zipping} and {\em unzipping} respectively.
\end{definition}

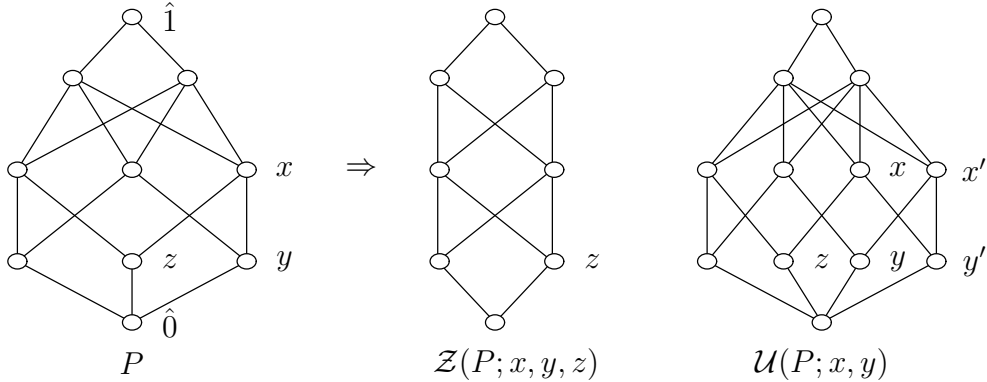
\begin{example}

\begin{center}
\unitlength 0.1in
\begin{picture}( 49.1000, 18.1700)( 13.5000,-23.1500)
%
\special{pn 8}%
\special{pa 2000 2184}%
\special{pa 2600 1864}%
\special{fp}%
\special{pa 2000 1864}%
\special{pa 2000 2184}%
\special{fp}%
\special{pa 1400 1864}%
\special{pa 2000 2184}%
\special{fp}%
%
\special{pn 8}%
\special{pa 1400 1384}%
\special{pa 1400 1864}%
\special{fp}%
\special{pa 1400 1864}%
\special{pa 2000 1384}%
\special{fp}%
\special{pa 2000 1384}%
\special{pa 2600 1864}%
\special{fp}%
\special{pa 2600 1864}%
\special{pa 2600 1384}%
\special{fp}%
\special{pa 2600 1384}%
\special{pa 2000 1864}%
\special{fp}%
\special{pa 2000 1864}%
\special{pa 1400 1384}%
\special{fp}%
%
\special{pn 8}%
\special{pa 1400 1384}%
\special{pa 1700 904}%
\special{fp}%
\special{pa 1700 904}%
\special{pa 2000 1384}%
\special{fp}%
\special{pa 2000 1384}%
\special{pa 2300 904}%
\special{fp}%
\special{pa 2300 904}%
\special{pa 1400 1384}%
\special{fp}%
%
\special{pn 8}%
\special{pa 2300 904}%
\special{pa 2600 1384}%
\special{fp}%
\special{pa 2600 1384}%
\special{pa 1700 904}%
\special{fp}%
%
\special{pn 8}%
\special{pa 1700 904}%
\special{pa 2000 584}%
\special{fp}%
\special{pa 2000 584}%
\special{pa 2310 904}%
\special{fp}%
%
\special{pn 8}%
\special{sh 0}%
\special{ar 2000 2184 50 40  0.0000000 6.2831853}%
%
\special{pn 8}%
\special{sh 0}%
\special{ar 2000 1864 50 40  0.0000000 6.2831853}%
%
\special{pn 8}%
\special{sh 0}%
\special{ar 1400 1864 50 40  0.0000000 6.2831853}%
%
\special{pn 8}%
\special{sh 0}%
\special{ar 2600 1864 50 40  0.0000000 6.2831853}%
%
\special{pn 8}%
\special{sh 0}%
\special{ar 2600 1384 50 40  0.0000000 6.2831853}%
%
\special{pn 8}%
\special{sh 0}%
\special{ar 2000 1384 50 40  0.0000000 6.2831853}%
%
\special{pn 8}%
\special{sh 0}%
\special{ar 1400 1384 50 40  0.0000000 6.2831853}%
%
\special{pn 8}%
\special{sh 0}%
\special{ar 2000 584 50 40  0.0000000 6.2831853}%
%
\special{pn 8}%
\special{sh 0}%
\special{ar 2290 904 50 40  0.0000000 6.2831853}%
%
\special{pn 8}%
\special{sh 0}%
\special{ar 1690 904 50 40  0.0000000 6.2831853}%
\put(20.0000,-24.0000){\makebox(0,0){$P$}}%
\put(40.1000,-24.0000){\makebox(0,0){$\Z(P;x,y,z)$}}%
\put(28.0000,-13.8300){\makebox(0,0){$x$}}%
\put(28.0000,-18.6300){\makebox(0,0){$y$}}%
\put(22.0000,-18.6300){\makebox(0,0){$z$}}%
%
\special{pn 8}%
\special{pa 5010 1864}%
\special{pa 5010 1384}%
\special{fp}%
\special{pa 5010 1384}%
\special{pa 5410 1864}%
\special{fp}%
\special{pa 5410 1864}%
\special{pa 5810 1384}%
\special{fp}%
\special{pa 5810 1384}%
\special{pa 6210 1864}%
\special{fp}%
\special{pa 6210 1864}%
\special{pa 6210 1384}%
\special{fp}%
\special{pa 6210 1384}%
\special{pa 5810 1864}%
\special{fp}%
\special{pa 5810 1864}%
\special{pa 5410 1384}%
\special{fp}%
\special{pa 5410 1384}%
\special{pa 5010 1864}%
\special{fp}%
%
\special{pn 8}%
\special{pa 5010 1864}%
\special{pa 5610 2184}%
\special{fp}%
\special{pa 5610 2184}%
\special{pa 5410 1864}%
\special{fp}%
\special{pa 5810 1864}%
\special{pa 5610 2184}%
\special{fp}%
\special{pa 5610 2184}%
\special{pa 6210 1864}%
\special{fp}%
%
\special{pn 8}%
\special{pa 6210 1384}%
\special{pa 5810 904}%
\special{fp}%
\special{pa 5810 904}%
\special{pa 5810 1384}%
\special{fp}%
\special{pa 5810 1384}%
\special{pa 5410 904}%
\special{fp}%
\special{pa 5410 904}%
\special{pa 5410 1384}%
\special{fp}%
\special{pa 5410 1384}%
\special{pa 5810 904}%
\special{fp}%
\special{pa 5810 904}%
\special{pa 5010 1384}%
\special{fp}%
\special{pa 5010 1384}%
\special{pa 5410 904}%
\special{fp}%
%
\special{pn 8}%
\special{pa 5410 904}%
\special{pa 5610 584}%
\special{fp}%
\special{pa 5610 584}%
\special{pa 5810 904}%
\special{fp}%
%
\special{pn 8}%
\special{sh 0}%
\special{ar 5610 584 50 40  0.0000000 6.2831853}%
%
\special{pn 8}%
\special{sh 0}%
\special{ar 5810 904 50 40  0.0000000 6.2831853}%
%
\special{pn 8}%
\special{sh 0}%
\special{ar 5010 1384 50 40  0.0000000 6.2831853}%
%
\special{pn 8}%
\special{sh 0}%
\special{ar 5410 1384 50 40  0.0000000 6.2831853}%
%
\special{pn 8}%
\special{sh 0}%
\special{ar 5810 1384 50 40  0.0000000 6.2831853}%
%
\special{pn 8}%
\special{sh 0}%
\special{ar 6210 1864 50 40  0.0000000 6.2831853}%
%
\special{pn 8}%
\special{sh 0}%
\special{ar 5810 1864 50 40  0.0000000 6.2831853}%
%
\special{pn 8}%
\special{sh 0}%
\special{ar 5410 1864 50 40  0.0000000 6.2831853}%
%
\special{pn 8}%
\special{sh 0}%
\special{ar 5010 1864 50 40  0.0000000 6.2831853}%
%
\special{pn 8}%
\special{sh 0}%
\special{ar 5610 2184 50 40  0.0000000 6.2831853}%
\put(56.1000,-24.0000){\makebox(0,0){$\U(P;x,y)$}}%
\put(44.1000,-18.6300){\makebox(0,0){$z$}}%
\put(56.1000,-18.6300){\makebox(0,0){$z$}}%
\put(60.1000,-13.8300){\makebox(0,0){$x$}}%
\put(60.1000,-18.6300){\makebox(0,0){$y$}}%
\put(64.1000,-13.8300){\makebox(0,0){$x'$}}%
\put(64.1000,-18.6300){\makebox(0,0){$y'$}}%
\put(22.0000,-21.8300){\makebox(0,0){$\hat 0$}}%
\put(22.0000,-5.8300){\makebox(0,0){$\hat 1$}}%
\put(32.0000,-13.8300){\makebox(0,0){$\Rightarrow$}}%
%
\special{pn 8}%
\special{pa 3600 1864}%
\special{pa 3900 2184}%
\special{fp}%
\special{pa 3900 2184}%
\special{pa 4200 1864}%
\special{fp}%
\special{pa 4200 1864}%
\special{pa 4200 1384}%
\special{fp}%
\special{pa 4200 1384}%
\special{pa 3600 1864}%
\special{fp}%
\special{pa 3600 1864}%
\special{pa 3600 1384}%
\special{fp}%
\special{pa 3600 1384}%
\special{pa 4200 1864}%
\special{fp}%
%
\special{pn 8}%
\special{pa 4200 1384}%
\special{pa 4200 904}%
\special{fp}%
\special{pa 4200 904}%
\special{pa 3600 1384}%
\special{fp}%
\special{pa 3600 1384}%
\special{pa 3600 904}%
\special{fp}%
\special{pa 3600 904}%
\special{pa 4200 1384}%
\special{fp}%
%
\special{pn 8}%
\special{pa 4200 904}%
\special{pa 3900 584}%
\special{fp}%
\special{pa 3900 584}%
\special{pa 3600 904}%
\special{fp}%
%
\special{pn 8}%
\special{sh 0}%
\special{ar 3610 904 50 40  0.0000000 6.2831853}%
%
\special{pn 8}%
\special{sh 0}%
\special{ar 3610 1384 50 40  0.0000000 6.2831853}%
%
\special{pn 8}%
\special{sh 0}%
\special{ar 4210 1384 50 40  0.0000000 6.2831853}%
%
\special{pn 8}%
\special{sh 0}%
\special{ar 4210 1864 50 40  0.0000000 6.2831853}%
%
\special{pn 8}%
\special{sh 0}%
\special{ar 3610 1864 50 40  0.0000000 6.2831853}%
%
\special{pn 8}%
\special{sh 0}%
\special{ar 4210 904 50 40  0.0000000 6.2831853}%
%
\special{pn 8}%
\special{sh 0}%
\special{ar 3900 584 50 40  0.0000000 6.2831853}%
%
\special{pn 8}%
\special{sh 0}%
\special{ar 3900 2184 50 40  0.0000000 6.2831853}%
%
\special{pn 8}%
\special{pa 6200 1384}%
\special{pa 5400 904}%
\special{fp}%
%
\special{pn 8}%
\special{sh 0}%
\special{ar 5410 904 50 40  0.0000000 6.2831853}%
%
\special{pn 8}%
\special{sh 0}%
\special{ar 6210 1384 50 40  0.0000000 6.2831853}%
\end{picture}%
\end{center}
\medskip
\end{example}

\begin{remark}
In general $\Z(P;x,y,z)$ may not be graded.
However, $\Z(P;x,y,z)$ is graded if $P$ is {\em thin},
namely, if for every $x \geq y$ in $P$ with $r(x)-r(y)=2$
the closed interval $[x,y]$ is a Boolean algebra of rank $2$.
See \cite[Proposition 4.4]{Reading}.
\end{remark}

In the rest of this section,
we study basic properties of zipping and unzipping.
Let $\Delta$ be a simplicial complex on the vertex set $V$.
The \textit{link of $F \in \Delta$ in $\Delta$}
is the simplicial complex
$$
\lk_\Delta(F) = \{ G \subseteq V \setminus F : G \cup F \in \Delta\}.
$$

\begin{definition}
Let $\Delta$ be a simplicial complex and let $\{i,j \}$ be an edge of $\Delta$.
The {\em (stellar) edge subdivision of $\Delta$ with respect to $\{i,j\}$} is the simplicial complex
$$
\{ F \in \Delta: F \not \supset \{i,j\}\}
\cup \{ F\cup\{v\},\ F\cup\{i,v\},\ F \cup \{j,v\}: F \in \lk_\Delta(\{i,j\})\},
$$
where $v$ is a new vertex.
The \textit{edge contraction of $i$ to $j$ in $\Delta$}
is the simplicial complex $\Delta'$ which is obtained
from $\Delta$ by identifying the vertices $i$ and $j$,
in other words,
$$\Delta' =
\{ F \in \Delta: i \not \in F\} \cup \{ (F \setminus \{i\}) \cup \{j\}:
i \in F \in \Delta\}.
$$
\end{definition}

\begin{proposition}\label{prop:unzip-edge sd}
With the same notation as in Definition \ref{def:unzip},
\begin{itemize}
\item[(1)] $\Z(\U(P;x,y);x',y',y)=P$.
\item[(2)] $\sd (\Z(P;x,y,z))$ is obtained from $\sd (P)$ by two successive edge contractions: first contract $y$ to $x$, then contract $x$ to $z$.
\item[(3)] $\sd (\U(P;x,y))$ is obtained from $\sd (P)$ by two successive edge subdivisions: first subdivide $\{x,y\}$ by $x'$, then subdivide $\{x,x'\}$ by $y'$.
\end{itemize}
\end{proposition}

\begin{proof}
Part (1) follows directly from Definition \ref{def:unzip}.

By the definition of the edge contraction,
contracting $y$ to $x$ in $\sd(P)$ and then contracting $x$ to $z$ results in the simplicial complex $\Delta$ obtained from $\sd(P)$ by replacing $x$ and $y$ by $z$ in all simplices (and deleting repetitions, of vertices in a simplex, and of simplices).
As (i) $z \in \Z(P;x,y,z)$ covers the same elements as $y$ and $z$ in $P$,
(ii) $w>z$ in $\Z(P;x,y,z)$ if and only if $w \ne x$ and either $w>y$ or $w>z$ in $P$,
and (iii) $x$ covers only $y,z\in P$, we conclude that the simplices of $\Delta$ are exactly the chains in $\Z(P;x,y,z)-\{\hat{0},\hat{1}\}$, proving (2).

To prove (3), let $\Gamma$ be the simplicial complex obtained from $\sd (P)$ by two successive edge subdivisions: first subdivide $\{x,y\}$ by $x'$, then subdivide $\{x,x'\}$ by $y'$.
The maximal simplices of $\Gamma$ are obtained from those of $\sd (P)$ by replacing each $x,y\in F\in \sd(P)$ in three ways, by either $x,y'$, or $y',x'$, or $x',y$; thus a maximal simplex $F\in \sd(P)$ corresponds to 3 simplices in $\Gamma$. Thus, by construction, these are exactly the maximal chains in $\U(P;x,y)-\{\hat{0},\hat{1}\}$.
\end{proof}

We say that a simplicial complex $\Delta$ satisfies the \textit{Link condition
with respect to an edge $\{i,j\} \in \Delta$} if
$\lk_\Delta(\{i\}) \cap \lk_\Delta(\{j\})  = \lk_\Delta(\{i,j\})$.
As the edge contractions in Proposition \ref{prop:unzip-edge sd} satisfy the Link Condition w.r.t.\ $\{x,y\}$ and $\{x,z\}$ in the corresponding simplicial complexes,
they preserve the PL-type for simplicial spheres \cite[Theorem 1.4]{Nevo-VK}, and preserve being a homology sphere \cite[Proposition 2.3]{Nevo-Novinsky}.
Hence we conclude the following result, obtained by Reading \cite[Theorem 4.7]{Reading} for the case of zipping in Gorenstein* posets \cite{Reading}.
\begin{corollary}\label{cor:zip Gorenstien}
If $P$ is a Gorenstein* poset (or a CW-sphere)
then so are $\U(P;x,y)$ and $\Z(P;x,y,z)$.
\end{corollary}

We use the following formula, observed by Reading \cite[Theorem 4.6]{Reading}, later on.
\begin{lemma}[Reading]\label{obs:zip cd-index}
The $\cd$-index changes under zipping as follows:
$$\Phi_P(\cc,\dd)=\Phi_{\Z(P;x,y,z)}(\cc,\dd) + \Phi_{[\hat{0},y]}(\cc,\dd) \cdot \dd \cdot \Phi_{[x,\hat{1}]}(\cc,\dd).$$
Thus, for unzipping we get
$$\Phi_{\U(P;x,y)}(\cc,\dd)=\Phi_P(\cc,\dd)+ \Phi_{[\hat{0},y]} (\cc,\dd) \cdot \dd \cdot \Phi_{[x,\hat{1}]}(\cc,\dd).$$
\end{lemma}

\section{$\cd$-indices of Gorenstein* posets of rank $5$}\label{sec:cd-5}

In this section,
we prove our first main result, Theorem \ref{thm:cd-rank5}.

We first recall the join of two posets.
Let $P$ and $Q$ be posets with $\hat 0$ and $\hat 1$.
The {\em join} $P*Q$ is the poset on the set $(P-\{\hat 1\}) \cup (Q -\{\hat 0\})$ with $x \leq y$ if either
(i) $x \leq y$ in $P$, (ii) $ x \leq y$ in $Q$, or (iii) $x \in P$ and $y \in Q$.
It is not hard to see that $\sd(P*Q)$ is the join of $\sd(P)$
and $\sd(Q)$ (as simplicial complexes).
Thus, if $P$ and $Q$ are Gorenstein* then so is $P*Q$.
The following formula was given in \cite[Lemma 1.1]{Stanley-cd}.

\begin{lemma}
\label{3.1}
If $P$ and $Q$ are Gorenstein* posets then
$\Phi_{P*Q}(\cc,\dd)=\Phi_P(\cc,\dd)\Phi_Q(\cc,\dd)$.
\end{lemma}

In the rest of this section,
we focus on Gorenstein* posets of rank $5$.
For a Gorenstein* poset of rank $5$,
we write its $\cd$-index in the form
$$\Phi_P(\cc,\dd)=\cc^4+ \alpha_1(P) \dd \cc^2 + \alpha_2(P) \cc\dd\cc + \alpha_3(P) \cc^2 \dd + \alpha_{13}(P) \dd^2.$$
We often use the following formula.
(We refer the readers for verification of the formula).
\begin{eqnarray}
\label{edgeformula}
f_{\{2,3\}}(P)=\alpha_{13}(P)+2(\alpha_1(P)+\alpha_2(P)+\alpha_3(P)) +4.
\end{eqnarray}

We first study necessary conditions of $\cd$-indices of Gorenstein* posets of rank $5$.
The following results were shown in \cite[Propositions 4.4 and 4.6]{Murai-Nevo-S*}.

\begin{lemma}
\label{3.2}
If $P$ is a Gorenstein* poset of rank $5$ then
$\alpha_{13}(P) \leq \alpha_1(P) \alpha_3(P)$.
\end{lemma}

\begin{lemma}
\label{3.3}
Let $P$ be a Gorenstein* poset of rank $5$  with $\alpha_2(P)=0$.
Then there are Gorenstein* posets $Q_1$ and $Q_2$ of rank $3$
such that $P=Q_1*Q_2$.
In particular, $\alpha_{13}(P) = \alpha_1(P) \alpha_3(P)$.
\end{lemma}

The next result gives a new restriction on $\cd$-indices.

\begin{lemma}
\label{3.4}
Let $P$ be a Gorenstein* poset of rank $5$.
If $\alpha_2(P)=1$
then
there are non-negative integers $b_1,b_2,b_3,c_1,c_2,c_3$ such that
$\alpha_1(P)=b_1+b_2+b_3$,
$\alpha_3(P)=c_1+c_2+c_3$ and $\alpha_{13}(P)=\alpha_1(P)\alpha_3(P) - (b_1 c_1 + b_2 c_2 + b_3c_3)$.
\end{lemma}

\begin{proof}
Consider the subposet $Q=\{\sigma \in P: r(\sigma)  \in \{1,2\}\}$,
where $r : P \to \ZZ_{\geq 0}$ is the rank function of $P$.
Since $(\hat 0,x )$ is the face poset of a cycle for any rank $3$ element $x \in P$,
the poset $Q$ is the face poset of a CW-complex $\Gamma$ which is a union of cycles.
Also, $Q$ is a Cohen-Macaulay poset since $Q$ is a rank selected subposet of $P$
(cf.\ \cite[III, Theorem 4.5]{StanleyGreenBook}),
which implies that the CW-complex $\Gamma$ is connected.
Similarly, let $Q'$ be the dual poset of $\{\sigma \in P: r(\sigma) \in \{3,4\}\}$.
Then, by considering the dual of the above argument,
it follows that $Q'$ is also the face poset of a connected CW-complex $\Gamma'$ which is a union of cycles.

Since $f_{\{1\}}(P)=2+\alpha_1(P)$ and $f_{\{2\}}(P)=2+\alpha_1(P)+\alpha_2(P)$,
$\alpha_2(P)=1$ implies that the number of edges in $\Gamma$ is equal to the number of vertices in $\Gamma$ plus $1$.
Since $\Gamma$ is a union of cycles,
this fact shows that $\Gamma$ is a union of two simple cycles $C$ and $C'$
such that the intersection of $C$ and $C'$ is either a point or a nontrivial simple path (namely, a simple path with at least one edge).

Suppose $C \cap C'$ is a point.
Then $C$ and $C'$ are the only simple cycles in $\Gamma$.
Thus,
for each rank $3$ element $x \in P$,
$(\hat 0,x)$ is equal to the face poset of $C$ or that of $C'$.
Then, for each rank $4$ element $y \in P$,
$y$ must cover exactly two elements $x$ and $x'$ with $(\hat 0,x)=(\hat 0,x')$
since $P$ is Gorenstein*.
However, this fact shows that $\Gamma'$ has two connected components,
which contradicts the connectedness of $\Gamma'$.
Hence $C \cap C'$ is a nontrivial simple path.

Let $C''$ be the cycle in $\Gamma$ obtained from $C \cup C'$ by removing the interior of $C \cap C'$.
Then, $C,C'$ and $C''$ are the only simple cycles in $\Gamma$,
and for each rank $3$ element $x \in P$,
$(\hat 0,x)$ must coincide with $C,C'$ or $C''$.
Let $c_1,c_2,c_3$ be the number of rank $3$ elements $x \in P$ such that
$(\hat 0,x)$ coincides with $C,C',C''$ respectively.
Let $b_1,b_2,b_3$ be the \emph{lengths} of the simple paths $C'\cap C'',C\cap C'',C\cap C'$ respectively, by means of number of edges. Clearly $b_i\geq 1$.
Also,
since $\Gamma'$ is connected,
by using the same argument as when we concluded that $C\cap C'$ is a nontrivial simple path,
we have $c_i\geq 1$.

Then, 
we have
\begin{align*}
f_{\{2,3\}}(P) &=
c_1(b_2+b_3)+c_2(b_1+b_3)+c_3(b_1+b_2)\\
&= (b_1+b_2+b_3)(c_1+c_2+c_3)-(b_1c_1+b_2c_2+b_3c_3).
\end{align*}
Set $b_i'=b_i-1$ and $c_i'=c_i-1$ for $i=1,2,3$.
Then,
since $\alpha_2(P)=1$,
$\alpha_1(P)=f_{\{2\}}(P)-3=b'_1+b'_2+b'_3$
and $\alpha_3(P)=f_{\{3\}}(P)-3=c'_1+c'_2+c'_3$.
By using \eqref{edgeformula} and the above equation,
a routine computation shows
$\alpha_{13}(P)=\alpha_1(P)\alpha_3(P) - (b'_1 c'_1 + b'_2 c'_2 + b'_3c'_3),$
as desired.
\end{proof}

Now we prove Theorem \ref{thm:cd-rank5}.
In the proof,
we use the following notation.
Let $P$ be a Gorenstein* poset and let $\sigma$ and $\tau$ be elements of $P \setminus\{\hat 0,\hat 1\}$ such that $\sigma$ covers $\tau$.
We say that $Q$ is obtained from $P$ by unzipping $(\sigma,\tau)$ $k$ times if $Q$ is obtained by the following successive process:
First, unzip $(\sigma,\tau)$ and consider $P'=\U(P;\sigma,\tau)$.
This unzipping creates new elements $\sigma'$ and $\tau'$ such that $\sigma'$ covers $\tau'$ in $P'$.
Next, unzip $(\sigma',\tau')$ in  $P'$,
and consider $P''=\U(P';\sigma',\tau')$.
Again, we obtain new elements $\sigma''$ and $\tau''$ such that $\sigma''$ covers $\tau''$ in $P''$,
and continue this procedure $k$ times.
Note that,
by Lemma \ref{obs:zip cd-index},
we have
$\Phi_Q(\cc,\dd)=\Phi_P (\cc,\dd) + k\cdot \Phi_{[\hat{0},\tau]}(\cc,\dd)\cdot \dd \cdot \Phi_{[\sigma,\hat{1}]}(\cc,\dd).$

\begin{proof}[Proof of Theorem \ref{thm:cd-rank5}]
The necessity follows from Lemmas \ref{3.2}, \ref{3.3} and \ref{3.4}.
We prove the sufficiency.
Let $C_k$ be the Gorenstein* poset of  rank $3$ corresponding to a cycle of length $k$
(i.e.\ $C_k -\{\hat 0,\hat 1\}$ is the face poset of a cycle of length $k$).
\medskip

(i)
Observe that $\Phi_{C_k}(\cc,\dd)=\cc^2 + (k-2) \dd$.
Then,
for all non-negative integers $\alpha_1$ and $\alpha_3$,
the join $C_{\alpha_1+2}* C_{\alpha_3+2}$ is a Gorenstein* poset of rank $5$ with
the desired $\cd$-index
$$
\Phi_{C_{\alpha_1+2}* C_{\alpha_3+2}}(\cc,\dd)
= \Phi_{C_{\alpha_1+2}} (\cc,\dd) \cdot \Phi_{C_{\alpha_3+2}}(\cc,\dd)
= \cc^4 + \alpha_1 \dd\cc^2+\alpha_3 \cc^2 \dd + \alpha_{1}\alpha_3 \dd^2.
$$

(ii)
Let $Q=\hat{B}_2 * C_3*\hat{B}_2$ described in figure (a),
where $\hat{B}_2$ is the Boolean algebra of rank $2$.
Note that $\Phi_Q(\cc,\dd)=\cc^4+ \cc\dd\cc$ by Lemma \ref{3.1}.

Let $R$ be the Gorenstein* poset obtained from $Q$ by unzipping
$(\tau_i,\rho)$ $b_i$ times for $i=1,2,3$ and by
unzipping $(\pi,\sigma_i)$ $c_i$ times for $i=1,2,3$.
We claim that $R$ has the desired $\cd$-index.
By Lemma \ref{obs:zip cd-index},
$\alpha_1(R)=b_1+b_2+b_3$,
$\alpha_2(R)=1$
and
$\alpha_3(R)=c_1+c_2+c_3$.
It remains to prove $\alpha_{13}(R)=\alpha_1(R)\alpha_3(R)-(b_1c_1+b_2c_2+b_3c_3)$.
\begin{center}
\unitlength 0.1in
\begin{picture}( 50.6900, 20.2200)( 25.2100,-32.1500)
%
\special{pn 8}%
\special{pa 3722 2402}%
\special{pa 4002 2852}%
\special{fp}%
%
\special{pn 8}%
\special{pa 3440 2852}%
\special{pa 3160 2402}%
\special{fp}%
\special{pa 3160 2402}%
\special{pa 3992 2852}%
\special{fp}%
\put(29.7100,-19.5200){\makebox(0,0){$\sigma_1$}}%
\put(35.3300,-19.5200){\makebox(0,0){$\sigma_2$}}%
\put(44.7000,-19.5200){\makebox(0,0){$\sigma_3$}}%
\put(44.7000,-24.0100){\makebox(0,0){$\tau_1$}}%
\put(35.3300,-24.0100){\makebox(0,0){$\tau_2$}}%
\put(29.7100,-24.0100){\makebox(0,0){$\tau_3$}}%
\put(32.5200,-15.0300){\makebox(0,0){$\pi$}}%
\put(32.5200,-28.5100){\makebox(0,0){$\rho$}}%
%
\special{pn 8}%
\special{pa 4284 2402}%
\special{pa 4284 1952}%
\special{fp}%
\special{pa 4284 1952}%
\special{pa 3722 2402}%
\special{fp}%
\special{pa 3722 2402}%
\special{pa 3160 1952}%
\special{fp}%
\special{pa 3160 1952}%
\special{pa 3160 2402}%
\special{fp}%
\special{pa 3160 2402}%
\special{pa 3722 1952}%
\special{fp}%
\special{pa 3722 1952}%
\special{pa 4284 2402}%
\special{fp}%
%
\special{pn 8}%
\special{pa 4264 2402}%
\special{pa 4002 2852}%
\special{fp}%
%
\special{pn 8}%
\special{pa 3722 2402}%
\special{pa 3440 2852}%
\special{fp}%
\special{pa 3440 2852}%
\special{pa 4284 2402}%
\special{fp}%
%
\special{pn 8}%
\special{pa 4284 1952}%
\special{pa 4002 1504}%
\special{fp}%
\special{pa 4002 1504}%
\special{pa 3722 1952}%
\special{fp}%
\special{pa 3722 1952}%
\special{pa 3440 1504}%
\special{fp}%
\special{pa 3440 1504}%
\special{pa 3160 1952}%
\special{fp}%
%
\special{pn 8}%
\special{pa 3160 1952}%
\special{pa 4002 1504}%
\special{fp}%
%
\special{pn 8}%
\special{pa 4284 1952}%
\special{pa 3440 1504}%
\special{fp}%
%
\special{pn 8}%
\special{pa 3722 3076}%
\special{pa 4002 2852}%
\special{fp}%
\special{pa 3722 3076}%
\special{pa 3440 2852}%
\special{fp}%
\special{pa 3722 1278}%
\special{pa 3440 1504}%
\special{fp}%
\special{pa 4002 1504}%
\special{pa 3722 1278}%
\special{fp}%
%
\special{pn 8}%
\special{sh 0}%
\special{ar 3722 1278 46 38  0.0000000 6.2831853}%
%
\special{pn 8}%
\special{sh 0}%
\special{ar 3722 1952 46 36  0.0000000 6.2831853}%
%
\special{pn 8}%
\special{sh 0}%
\special{ar 4284 1952 46 36  0.0000000 6.2831853}%
%
\special{pn 8}%
\special{sh 0}%
\special{ar 4284 2402 46 38  0.0000000 6.2831853}%
%
\special{pn 8}%
\special{sh 0}%
\special{ar 3722 2402 46 38  0.0000000 6.2831853}%
%
\special{pn 8}%
\special{sh 0}%
\special{ar 3160 2402 48 38  0.0000000 6.2831853}%
%
\special{pn 8}%
\special{sh 0}%
\special{ar 3160 1952 48 38  0.0000000 6.2831853}%
%
\special{pn 8}%
\special{sh 0}%
\special{ar 3722 3076 46 36  0.0000000 6.2831853}%
%
\special{pn 8}%
\special{sh 0}%
\special{ar 4002 2852 46 38  0.0000000 6.2831853}%
%
\special{pn 8}%
\special{sh 0}%
\special{ar 3440 2852 48 38  0.0000000 6.2831853}%
%
\special{pn 8}%
\special{sh 0}%
\special{ar 3440 1504 48 38  0.0000000 6.2831853}%
%
\special{pn 8}%
\special{sh 0}%
\special{ar 4002 1504 46 38  0.0000000 6.2831853}%
\put(37.2100,-33.0000){\makebox(0,0){figure (a)}}%
\put(39.0800,-30.7600){\makebox(0,0){$\hat 0$}}%
\put(39.0800,-12.7800){\makebox(0,0){$\hat 1$}}%
%
\special{pn 8}%
\special{pa 5558 2402}%
\special{pa 5558 1952}%
\special{fp}%
\special{pa 5558 1952}%
\special{pa 6120 2402}%
\special{fp}%
\special{pa 6120 2402}%
\special{pa 6120 1952}%
\special{fp}%
\special{pa 6120 1952}%
\special{pa 6682 2402}%
\special{fp}%
\special{pa 6682 2402}%
\special{pa 6682 1952}%
\special{fp}%
\special{pa 6682 1952}%
\special{pa 7244 2402}%
\special{fp}%
\special{pa 7244 1952}%
\special{pa 5558 2402}%
\special{fp}%
\special{pa 7244 2402}%
\special{pa 7244 1952}%
\special{fp}%
\special{pa 7244 1952}%
\special{pa 6870 1504}%
\special{fp}%
\special{pa 6870 1504}%
\special{pa 6682 1952}%
\special{fp}%
\special{pa 6682 1952}%
\special{pa 5932 1504}%
\special{fp}%
\special{pa 5932 1504}%
\special{pa 7244 1952}%
\special{fp}%
\special{pa 6120 1952}%
\special{pa 5932 1504}%
\special{fp}%
\special{pa 5932 1504}%
\special{pa 5558 1952}%
\special{fp}%
\special{pa 5558 1952}%
\special{pa 6870 1504}%
\special{fp}%
\special{pa 6870 1504}%
\special{pa 6120 1952}%
\special{fp}%
%
\special{pn 8}%
\special{pa 6870 2852}%
\special{pa 7244 2402}%
\special{fp}%
\special{pa 7244 2402}%
\special{pa 5932 2852}%
\special{fp}%
\special{pa 5932 2852}%
\special{pa 6682 2402}%
\special{fp}%
\special{pa 6682 2402}%
\special{pa 6870 2852}%
\special{fp}%
\special{pa 6870 2852}%
\special{pa 6120 2402}%
\special{fp}%
\special{pa 6120 2402}%
\special{pa 5932 2852}%
\special{fp}%
\special{pa 5932 2852}%
\special{pa 5558 2402}%
\special{fp}%
\special{pa 5558 2402}%
\special{pa 6870 2852}%
\special{fp}%
%
\special{pn 8}%
\special{sh 0}%
\special{ar 5558 2402 48 38  0.0000000 6.2831853}%
%
\special{pn 8}%
\special{sh 0}%
\special{ar 6120 2402 48 38  0.0000000 6.2831853}%
%
\special{pn 8}%
\special{sh 0}%
\special{ar 6682 2402 48 38  0.0000000 6.2831853}%
%
\special{pn 8}%
\special{sh 0}%
\special{ar 7244 2402 48 38  0.0000000 6.2831853}%
%
\special{pn 8}%
\special{sh 0}%
\special{ar 7244 1952 48 38  0.0000000 6.2831853}%
%
\special{pn 8}%
\special{sh 0}%
\special{ar 6682 1952 48 38  0.0000000 6.2831853}%
%
\special{pn 8}%
\special{sh 0}%
\special{ar 6120 1952 48 38  0.0000000 6.2831853}%
%
\special{pn 8}%
\special{sh 0}%
\special{ar 5558 1952 48 38  0.0000000 6.2831853}%
\put(57.4500,-15.0300){\makebox(0,0){$\pi$}}%
\put(57.4500,-28.5100){\makebox(0,0){$\rho$}}%
\put(53.7000,-24.0100){\makebox(0,0){$\tau_1$}}%
\put(53.7000,-19.5200){\makebox(0,0){$\sigma_1$}}%
\put(59.3200,-19.5200){\makebox(0,0){$\sigma_2$}}%
\put(64.9400,-19.5200){\makebox(0,0){$\sigma_3$}}%
\put(74.3100,-19.5200){\makebox(0,0){$\sigma_4$}}%
\put(59.3200,-24.0100){\makebox(0,0){$\tau_2$}}%
\put(64.9400,-24.0100){\makebox(0,0){$\tau_3$}}%
\put(74.3100,-24.0100){\makebox(0,0){$\tau_4$}}%
%
\special{pn 8}%
\special{pa 6400 3076}%
\special{pa 6870 2852}%
\special{fp}%
%
\special{pn 8}%
\special{pa 6870 1504}%
\special{pa 6400 1278}%
\special{fp}%
\special{pa 6400 1278}%
\special{pa 5932 1504}%
\special{fp}%
%
\special{pn 8}%
\special{pa 5932 2852}%
\special{pa 6400 3076}%
\special{fp}%
%
\special{pn 8}%
\special{sh 0}%
\special{ar 6400 1278 46 38  0.0000000 6.2831853}%
%
\special{pn 8}%
\special{sh 0}%
\special{ar 6400 3076 48 38  0.0000000 6.2831853}%
%
\special{pn 8}%
\special{sh 0}%
\special{ar 6870 1504 46 38  0.0000000 6.2831853}%
%
\special{pn 8}%
\special{sh 0}%
\special{ar 5932 1504 46 38  0.0000000 6.2831853}%
%
\special{pn 8}%
\special{sh 0}%
\special{ar 5932 2852 48 38  0.0000000 6.2831853}%
%
\special{pn 8}%
\special{sh 0}%
\special{ar 6870 2852 48 38  0.0000000 6.2831853}%
\put(64.1200,-32.9800){\makebox(0,0){figure (b)}}%
\end{picture}%
\medskip
\end{center}

For $i=1,2,3$,
let $B_i$ be the set of rank $2$ elements of $R$ that consists of $\tau_i$
and elements which are added when we unzip $(\tau_i,\rho)$ $b_i$ times,
and define the set $C_i$ of rank $3$ elements of $R$ with $\sigma_i \in C_i$ similarly.
Then,
since each element of $C_i$ exactly covers all rank $2$ elements which are not in $B_i$,
\begin{align*}
f_{\{2,3\}}(R)& \!=\! |C_1|(|B_2|+|B_3|)+|C_2|(|B_1|+|B_3|)+|C_3|(|B_1|+|B_2|)\\
&\! =\!  (|B_1|\!+\!|B_2|\!+\!|B_3|)(|C_1|\!+\!|C_2|\!+\!|C_3|)-(|B_1||C_1|\!+\!|B_2||C_2|\!+\!|B_3||C_3|).
\end{align*}
Observe $b_i=|B_i|-1$ and $c_i=|C_i|-1$ for $i=1,2,3$.
Then, in the same way as in the proof of Lemma \ref{3.4},
\eqref{edgeformula} and routine computations guarantee
$\alpha_{13}(R)=\alpha_1(R)\alpha_3(R)-(b_1c_1+b_2c_2+b_3c_3)$.
\medskip

(iii)
Let $\alpha_1,\alpha_2,\alpha_3$ and $\alpha_{13}$ be non-negative integers with $\alpha_2 \geq 2$ and $\alpha_{13} \leq \alpha_1\alpha_3$.
Let $Q'=\hat{B}_2* C_4*\hat{B}_2$ described in figure (b).
Note that $\Phi_Q(\cc,\dd) = \cc^4 + 2 \cc\dd\cc$.
\medskip

{\em Case 1.}
Suppose $\alpha_1=0$. Then $\alpha_{13} =0$ by the assumption.
Let $P$ be the Gorenstein* poset obtained from $Q'$ by unzipping $(\pi,\sigma_1)$  $\alpha_3$ times and by
unzipping $(\sigma_1,\tau_1)$ $(\alpha_2 -2)$ times.
Lemma \ref{obs:zip cd-index} shows that $P$ has the desired $\cd$-index
$\cc^4+\alpha_2\cc\dd\cc+\alpha_3 \cc^2\dd.$

{\em Case 2.}
Suppose $\alpha_1>0$ and $\alpha_{13} \leq \alpha_1$.
Let $R'$ be the Gorenstein* poset obtained from $Q'$ by applying the following unzipping.
\begin{itemize}
\item[(U1)]
unzip $(\tau_1,\rho)$ $\alpha_{13}$ times;
\item[(U2)]
unzip $(\tau_4,\rho)$ $(\alpha_1-\alpha_{13})$ times;
\item[(U3)]
unzip $(\pi,\sigma_1)$;
\item[(U4)]
unzip $(\pi,\sigma_2)$ $(\alpha_3-1)$ times.
\end{itemize}
By Lemma \ref{obs:zip cd-index},
$\alpha_1(R')=\alpha_{13} +(\alpha_1-\alpha_{13})=\alpha_1$,
$\alpha_2(R')=2$
and
$\alpha_3(R')=1 + (\alpha_3 -1)=\alpha_3$.
We claim $\alpha_{13}(R')=\alpha_{13}$.

Let $B_1$ be the set of rank 2 elements of $R'$ consisting of $\tau_1$ and elements which are added by unzipping (U1),
and let $B_4$ be that consisting of $\tau_4$ and elements added by (U2).
Similarly,
let $C_1$ be the set of rank 3 elements of $R'$ consisting of $\sigma_1$ and the element added by (U3),
and let $C_2$ be that consisting of $\tau_2$ and elements added by (U4).
Finally, let $B_2=\{\tau_2\}$, $B_3=\{\tau_3\}$,
$C_3=\{\sigma_3\}$ and $C_4=\{\sigma_4\}$.

Then $B_1 \cup \cdots \cup B_4$ and $C_1 \cup \cdots \cup C_4$ are partitions of the sets of elements of $R'$ of ranks $2$ and $3$
respectively.
Also, elements in $C_i$ exactly cover elements in $B_i$ and $B_{i+1}$, where we consider $B_5=B_1$.
This implies
\begin{align*}
f_{\{2,3\}}(R) &= |C_1|(|B_1|\!+\!|B_2|)\!+\!|C_2|(|B_2|\!+\!|B_3|)\!+\!|C_3|(|B_3|\!+\!|B_4|)\!+\!|C_4|(|B_4|\!+\!|B_1|)\\
&=2 \cdot (\alpha_{13}+2) + \alpha_3 \cdot 2 + 1 \cdot (\alpha_1 -\alpha_{13}+2) + 1\cdot (\alpha_1+2) \\
&=\alpha_{13}+2(\alpha_1+2+\alpha_3)+4.
\end{align*}
Hence $\alpha_{13}(R)=\alpha_{13}$ by \eqref{edgeformula}.

Finally, consider the Gorenstein* poset $P$ obtained from $R'$
by unzipping $(\sigma_1,\tau_1)$ $(\alpha_2-2)$ times.
Then Lemma \ref{obs:zip cd-index} show that $P$ has the desired $\cd$-index
$$\Phi_{P}(\cc,\dd)=\Phi_{R'}(\cc,\dd)+ (\alpha_2-2) \cc\dd\cc=
\cc^4+ \alpha_1 \dd \cc^2 + \alpha_2 \cc\dd\cc + \alpha_3 \cc^2 \dd + \alpha_{13} \dd^2.$$

{\em Case 3.}
Suppose $\alpha_1>0$ and $\alpha_{13} > \alpha_1$.
Recall $\alpha_{13} \leq \alpha_1 \alpha_3$.
Let $\beta \leq \alpha_3$ be the integer satisfying
$$\alpha_1 (\beta-1) < \alpha_{13} \leq \alpha_1 \beta $$
and let
$$p= \alpha_1 \beta  - \alpha_{13}.$$
Note that $\beta \geq 2$ and $0 \leq p < \alpha_1$.
Let $R'$ be the Gorenstein* poset obtained from $Q'$ by applying the following unzipping.
\begin{itemize}
\item[(U1)] unzip $(\tau_1,\rho)$ $(\alpha_1-p)$ times
and unzip $(\tau_2,\rho)$ $p$ times;
\item[(U2)]
unzip $(\pi,\sigma_1)$ $(\beta-1)$ times,
unzip $(\pi,\sigma_3)$ $(\alpha_3-\beta)$ times, and
unzip $(\pi,\sigma_4)$.
\end{itemize}
Then, Lemma \ref{obs:zip cd-index} shows
$\alpha_1(R')=\alpha_1$,
$\alpha_2(R')=2$ and $\alpha_3(R')=\alpha_3$.
Also, a computation similar to Case 2 shows
\begin{align*}
f_{\{2,3\}}(R) &= \beta \cdot (\alpha_{1}+2) + 1 \cdot (p+2)  + (\alpha_3-\beta+1)\cdot 2 + 2 \cdot (\alpha_1-p+2) \\
&= (\beta\alpha_1-p) +2(\alpha_3+2 +\alpha_1) + 4.
\end{align*}
Since $\alpha_{13}=\beta \alpha_1 -p$,
the above equation and \eqref{edgeformula} show
$\alpha_{13}(R')=\alpha_{13}$.

Then the Gorenstein* poset $P$ obtained from $R'$
by unzipping $(\sigma_1,\tau_1)$ $(\alpha_2-2)$ times has the desired $\cd$-index.
\end{proof}

\section{$\dd$-vectors of Gorenstein* posets of rank $5$ and $6$}\label{sec:d-vectors}

In this section,
we study $\dd$-vectors of Gorenstein* posets of rank $5$ and $6$.
We often use the following obvious fact.

\begin{lemma}
Let $x$ and $y$ be non-negative integers.
Then $y \leq {x^2 \over 4}$ if and only if there are non-negative integers $a$ and $b$ such that
$y \leq ab$ and $a+b \leq x$.
\end{lemma}

We first classify $\dd$-vectors of Gorenstein* posets of rank $5$.

\begin{theorem}\label{thm:rank5}
The vector $(1,x,y) \in \ZZ_{\geq 0}^3$
is the $\mathbf{d}$-vector of a Gorenstein* poset of rank $5$ if and only if
it satisfies $y \leq {(x-1)^2 \over 4}$ or there are non-negative integers $a$ and $b$ such that $x=a+b$ and $y=ab$.
\end{theorem}

\begin{proof}
(Necessity).
Let $P$ be a Gorenstein* poset and $\dd(P)=(1,x,y)$.
Suppose $y > {(x-1)^2 \over 4}$.
We show that there are non-negative integers $a$ and $b$ such that $x=a+b$ and $y=ab$.

Observe that $x=\alpha_1(P)+\alpha_2(P)+\alpha_3(P)$ and $y=\alpha_{13}(P)$.
Then
$${(x-1)^2 \over 4}<y= \alpha_{13}(P) \leq \alpha_1(P) \alpha_3(P) \leq
{(\alpha_1(P)+\alpha_3(P))^2 \over 4} \leq {x^2 \over 4}.$$
This says that $\alpha_1(P)+\alpha_3(P)=x$ and therefore $\alpha_2(P)=0$.
Then Theorem \ref{thm:cd-rank5}(i) shows $y=\alpha_1(P)\alpha_3(P)$.

(Sufficiency).
For all non-negative integers $a$ and $b$,
$\dd(C_{a+2}*C_{b+2})=(1,a+b,ab)$ by Lemma \ref{3.1}.
Let $(1,x,y) \in \ZZ_{\geq 0}^3$ with $y \leq {(x-1)^2 \over 4}$.
What we must prove is that there is a Gorenstein* poset of rank $5$ such that $\dd(P)=(1,x,y)$.

Since $y \leq {(x-1)^2 \over 4}$,
there are non-negative integers $a$ and $b$ such that $y \leq ab$ and $a+b \leq x-1$.
We may choose these integers so that $a(b-1) < y \leq ab$.
Let $\alpha_1=a,$ $\alpha_2=x-a-b$, $\alpha_3=b$ and $\alpha_{13}=y$.
It is enough to show that there is a Gorenstein* poset of rank $5$ whose $\cd$-index is
$\cc^4+\alpha_1 \dd\cc^2+ \alpha_2 \cc\dd\cc+ \alpha_3\cc^2\dd+ \alpha_{13}\dd^2$.
If $\alpha_2 \geq 2$ then the existence of such a poset follows from Theorem \ref{thm:cd-rank5}(iii).
If $\alpha_2 <2$ then $\alpha_2=1$.
We claim that $\alpha_1,\alpha_3,\alpha_{13}$ satisfy the conditions in Theorem \ref{thm:cd-rank5}(ii).
If $b=0$ then the statement is obvious.
If $b>0$ then, since $0 \leq (ab-y) <a$, the partition of integers
$a=(ab-y)+0+(a-(ab-y))$ and $b=1+(b-1)+0$ shows that $\alpha_1,\alpha_3,\alpha_{13}$ satisfy the desired conditions.
\end{proof}

Next, we consider Gorenstein* posets of rank $6$.

\begin{theorem}\label{thm:rank6}
The vector $(1,x,y) \in \ZZ_{\geq 0}^3$
is the $\mathbf{d}$-vector of a Gorenstein* poset of rank $6$ if and only if
it satisfies $y \leq {x^2 \over 4}$.
\end{theorem}

\begin{proof}
(Necessity).
Let $P$ be a Gorenstein* poset of rank $6$ and $\dd(P)=(1,x,y)$.
We write the $\cd$-index of $P$ in the form
$$\Phi(P)=c^5+ \alpha_1\dd\cc^3 + \alpha_2\cc\dd\cc^2 + \alpha_3\cc^2\dd\cc + \alpha_4\cc^3\dd + \alpha_{13}\dd^2\cc
+ \alpha_{14}\dd\cc\dd + \alpha_{24}\cc\dd^2.$$

It was proved in \cite[Proposition 4.4]{Murai-Nevo-S*} that $\alpha_{13} \leq \alpha_1\alpha_3$, $\alpha_{14} \leq \alpha_1\alpha_4$ and $\alpha_{24} \leq \alpha_2\alpha_4$. Then, since $x=\alpha_1+\alpha_2+\alpha_3+\alpha_4$ and $y=\alpha_{13}+\alpha_{14}+ \alpha_{24}$,
we have
$$y\leq \alpha_1\alpha_3 + \alpha_1\alpha_4 + \alpha_2\alpha_4\leq (\alpha_1+\alpha_2)(\alpha_3+\alpha_4)\leq {x^2 \over 4},$$
as desired.

(Sufficiency).
Let $(1,x,y) \in \ZZ_{\geq 0}^3$ with $y \leq {x^2 \over 4}$.
We show that there is a Gorenstein* poset of rank $6$ whose $\dd$-vector is $(1,x,y)$.
Since $\dd(P)=\dd(P*\hat B_2)$ for any Gorenstein* poset $P$ of rank $5$,
by Theorem \ref{thm:rank5}, we may assume ${(x-1)^2 \over 4} < y \leq {x^2 \over 4}$.
Then there are positive integers $a$ and $b$ such that
$a(b-1) < y \leq ab$ and $a+b =x$.
Let
$r=ab-y$.
Then $0 \leq r <a$ and
$$(b-1)(a-r)+r(b-1)+(a-r)=y.$$

Let $Q=C_{a-r+2}*\hat{B}_2*C_{b+1}$.
By Lemma \ref{3.1}, $\dd(Q)=(1,a-r+b-1,(a-r)(b-1))$.
Let $\rho$ be a rank $2$ element of $P$,
$\pi$ a rank $4$ element of $P$,
and let $\sigma,\tau$ be the rank $3$ elements of $P$.
Note that $[\hat0,\rho]=[\pi,\hat1]=\hat{B}_2$,
$[\hat 0,\tau] =C_{a-r+2}$
and
$[\sigma,\hat 1] =C_{b+1}$.
Let $P$ be the Gorenstein* poset obtained from $Q$ by unzipping $(\sigma,\rho)$ $r$ times
and by unzipping $(\pi,\tau)$.
Then, by Lemma \ref{obs:zip cd-index}, we have
\begin{align*}
\dd(P)&=\dd(Q)+r(0,1,b-1)+(0,1,a-r)\\
&=(1,a+b,(a-r)(b-1)+r(b-1)+(a-r))\\
&=(1,x,y),
\end{align*}
as desired.
\end{proof}

\begin{remark}
Gal \cite{Gal} proved that $\gamma_2(\Delta) \leq {\gamma_1(\Delta)^2 \over 4}$
for any flag homology $4$-sphere $\Delta$
(see Section 5 for details).
This result and the relation between $\dd$-vectors and $\gamma$-vectors
give an alternative proof of the necessity of Theorem \ref{thm:rank6}.
\end{remark}

\begin{remark}\label{rem:polytopal-cd}
For the posets $P$ constructed to show sufficiency in Theorems \ref{thm:cd-rank5}, \ref{thm:rank5} and \ref{thm:rank6},
their order complexes $\sd (P)$ are \emph{polytopal}, namely, can be realized as the boundary of a polytope.

Indeed, it is not hard to see that $P$ is obtained from the join
$\hat B_2 * \hat B_2 * \cdots * \hat B_2$ of Boolean algebras of rank $2$ by applying unzipping repeatedly.
The order complex of the join of Boolean algebras of rank $2$ is the boundary of a cross polytope.
As edge subdivisions preserve polytopality (just place the new vertex \emph{beyond} the edge, cf.\  \cite[p.\ 78]{Ziegler}),
by Proposition \ref{prop:unzip-edge sd}(3) we conclude that $\sd (P)$ is polytopal.
\end{remark}



\section{Questions}\label{sec:Q}

Let $\Delta$ be an $(n-1)$-dimensional simplicial complex.
Recall that the {\em $f$-vector} $f(\Delta)=(1,f_0,f_1,\dots,f_{n-1})$ of $\Delta$ is defined by $f_i=|\{F \in \Delta: |F|=i+1\}|$,
and the {\em $h$-vector}
$h(\Delta)=(h_0,h_1,\dots,h_n)$ of $\Delta$ is defined by the relation
$$\sum_{i=0}^n h_ix^{n-i} = \sum_{i=0}^n f_{i-1} (x-1)^{n-i}.$$
If $\Delta$ is Gorenstein* then $h_i=h_{n-i}$ for all $i$ by the Dehn-Sommerville equations,
and in this case the $\gamma$-vector $\gamma(\Delta)=(\gamma_0,\gamma_1,\dots,\gamma_{\lfloor \frac n 2 \rfloor})$
of $\Delta$ is defined by the relation
$$\sum_{i=0}^n h_i x^{i} = \sum_{i=0}^{\lfloor \frac n 2 \rfloor} \gamma_i x^i(1+x)^{n-2i}.$$
Thus, if $\Delta$ is the order complex $\sd(P)$ of a Gorenstein* poset $P$,
$\gamma(\Delta)$ and $\dd(P)=(1,\delta_1,\dots,\delta_{\lfloor \frac n 2 \rfloor})$ are related by
$\gamma_i= 2^i \delta_i$ for all $i$.

A simplicial complex is said to be \textit{flag} if every minimal non-face has at most two elements.
The study of $\gamma$-vectors of flag homology spheres (namely, flag Gorenstein* complexes) is one of current trends in face enumeration.
On $\gamma$-vectors of flag homology spheres,
one of the most important open problems is the conjecture of Gal \cite[Conjecture 2.1.7]{Gal}
which states that the $\gamma$-vector of
a flag homology sphere is non-negative.
Gal's conjecture is known to be true in dimension $\leq 4$.
Moreover, in low dimension
Gal essentially proved the following result.

\begin{theorem}[Gal]
\label{rem:gamma-3,4}
Let $\Lambda_k$ be the set of $\gamma$-vectors of flag homology $k$-spheres.
\begin{itemize}
\item[(i)]
$\Lambda_3\! \supset \{(1,x,y)\in\ZZ_{\geq 0}^3: y\leq \frac{(x-1)^2}{4} \mbox{ or }(x,y)=(a+b,ab) \mbox{ for some }a,b \in \ZZ_{\geq 0}\}$.
\item[(ii)]
$\Lambda_4= \{(1,x,y)\in\ZZ_{\geq 0}^3: y\leq \frac{x^2}{4}\}$,
\end{itemize}
\end{theorem}

\begin{proof}
(i) is proved in \cite[Theorem 3.2.1]{Gal} and $\Lambda_4\subseteq \{(1,x,y)\in\ZZ_{\geq 0}^3: y\leq \frac{x^2}{4}\}$ is proved in \cite[Theorem 3.1.3]{Gal}.
It remains to show the reverse containment in (ii).

We sketch the proof since it is similar to that of Theorem \ref{thm:rank6}.
Let $(1,x,y) \in \ZZ_{\geq 0}^3$ with $y \leq {x^2 \over 4}$.
As taking suspension does not change the $\gamma$-vector, we may assume ${(x-1)^2 \over 4} < y \leq {x^2 \over 4}$; in particular $x\geq 2$ and $y\geq 1$.
Then there are $a,b,r \in \ZZ_{\geq 0}$ with $a,b\geq 1$ and $r <a$  such that $a+b=x$ and $ab-r =y$.
Consider the simplicial complex $K=\tilde C_{a-r+4}*\tilde{B}_2* \tilde C_{b+3}$,
where $\tilde C_k$ is the cycle ($1$-dimensional simplicial sphere) of length $k$
and where $\tilde{B}_2$ is the $0$-sphere with vertices $\{x,y\}$.
(Here we consider the join as simplicial complexes.)
Then, by multiplicativity of $\gamma$ w.r.t. joins (cf. \cite[Remark 2.1.9]{Gal}), $\gamma(K)=(1,a-r+b-1,(a-r)(b-1))$.
Let $\{s\} \in C_{a-r+4}$ and  $\{t\} \in C_{b+3}$.
If we subdivide the edge $\{x,s\}$ $r$-times and subdivide the edge $\{y,t\}$,
by \cite[Proposition 2.4.3]{Gal},
we obtain a flag $4$-sphere $\Delta$ with $\gamma(\Delta)=(1,x,y)$.
%
\end{proof}

Moreover,
Gal \cite[Conjecture 3.2.2]{Gal} conjectured the following, which, if true, gives a complete characterization of the $f$-vectors of flag homology $3$-spheres.

\begin{conjecture}[Gal]
Let $\Delta$ be a  flag homology $3$-sphere and let $\gamma(\Delta)=(1,\gamma_1,\gamma_2)$.
If $\gamma_2 > {(\gamma_1-1)^2 \over 2}$ then $\Delta$ is the join of two cycles.
\end{conjecture}

Note that Lemma \ref{3.3} and the proof of Theorem \ref{thm:rank5} show that the above conjecture is true for order complexes.
If the above conjecture is true then the inclusion in Theorem \ref{rem:gamma-3,4}(i)
becomes an equality.
These facts and Theorem \ref{thm:rank5,6d-vectors} suggest the following question.

\begin{question}\label{Q1}
Let $\mathcal D_k$ be the set of $\dd$-vectors of Gorenstein* posets of rank $k+2$.
Is there any relation between $\Lambda_k$ and $\mathcal D_k$?
Is it true that $\Lambda_k = \mathcal D_k$ for all $k$?
Or at least does $\Lambda_k$ contain $\mathcal D_k$?
\end{question}

Note that equality in Question \ref{Q1} would imply Gal's conjecture on the non-negativity of the $\gamma$-vector of flag spheres, by Karu's result on the non-negativity of the $\cd$-index of Gorenstein* posets \cite{Karu-cd}.
\\

\textbf{Acknowledgments}:
The second author would like to thank artblau Tanzwerkstatt Braunschweig and Cornell University math departnemt for the hospitality while working on this manuscript.

\bibliography{gbiblio}
\bibliographystyle{plain}

\end{document}